\documentclass[11pt]{amsart}
\usepackage{latexsym}
\usepackage{amsmath,amssymb,amsfonts}
\usepackage{graphics}
\usepackage[all]{xy}

\def\Bbb{\mathbb}

\def\eea{\end{eqnarray*}}

\newtheorem{thm}{Theorem}[section]
\newtheorem{prop}[thm]{Proposition}
\newtheorem{cor}[thm]{Corollary}
\newtheorem{lem}[thm]{Lemma}

\newenvironment{xpl}{\mbox{ }\\ {\bf  Example}\mbox{ }}{
\hfill $\diamondsuit$\mbox{}\bigskip}
\newenvironment{rmk}{\mbox{ }\\{\bf  Remark}\mbox{ }}{
\hfill $\Box$\mbox{}\bigskip}

\begin{document}
\renewcommand{\theequation}{\thesection.\arabic{equation}}

\title[$G$-monopole classes, Ricci flow, and Yamabe invariants]{$G$-monopole classes, Ricci flow, and Yamabe invariants of 4-manifolds}

\author{Chanyoung Sung}
\date{\today}

\address{Dept. of Mathematics and Institute for Mathematical Sciences \\
Konkuk University\\
         120 Neungdong-ro, Gwangjin-gu, Seoul, KOREA}
\email{cysung@kias.re.kr}
\thanks{This work was supported by the National Research Foundation of Korea(NRF) grant funded by the Korea government(MEST). (No. 2011-0002791, 2012-0000341)}

\begin{abstract}
On a smooth closed oriented $4$-manifold $M$ with a smooth action by a finite group $G$, we show that a $G$-monopole class gives the $L^2$-estimate of the Ricci curvature of a $G$-invariant Riemannian metric, and derive a topological obstruction to the existence of a $G$-invariant nonsingular solution to the normalized Ricci flow on $M$.

In particular, for certain $m$ and $n$, $m\Bbb CP_2 \# n\overline{\Bbb CP}_2$ admits an infinite family of topologically equivalent but smoothly distinct non-free actions of $\Bbb Z_d$ such that it admits no nonsingular solution to the normalized Ricci flow for any initial metric invariant under such an action, where $d>1$ is a non-prime integer.

We also compute the $G$-Yamabe invariants of some 4-manifolds with $G$-monopole classes and the oribifold Yamabe invariants of some 4-orbifolds.
\end{abstract}
\maketitle
\setcounter{section}{0}
\setcounter{equation}{0}

\section{Introduction}
This article is a continuation of our previous paper \cite{sung5} with a view to geometric applications.
On a smooth closed oriented 4-manifold $M$, an element of $H^2(M;\Bbb Z)$ is called  a \emph{monopole class} if it arises as the first Chern class of a Spin$^c$ structure of $M$ for which the Seiberg-Witten equations
$$\left\{
\begin{array}{ll} D_A\Phi=0\\
  F_{A}^+=\Phi\otimes\Phi^*-\frac{|\Phi|^2}{2}\textrm{Id},
\end{array}\right.
$$
admit a solution $(A,\Phi)$ for every choice of a Riemannian metric on $M$. It is well-known that a monopole class  gives a lower bound of the $L^2$-norm of various curvatures for any Riemannian metric, and hence a necessary condition for the existence of an Einstein metric and the Yamabe invariant of the manifold can be obtained.

In order to detect a monopole class, one needs to compute Seiberg-Witten invariants gotten by the intersection theory on the moduli space of solutions of the Seiberg-Witten equations or further refined Bauer-Furuta invariant given by the stably-framed bordism class of the moduli space, which is equivalent to the homotopy class of the Seiberg-Witten equations as a map between configuration spaces with Sobolev norms. But in many important cases, those invariants are difficult to compute or get trivial.

In the meantime, sometimes we need a solution of the
Seiberg-Witten equations for a specific metric rather than any
Riemannian metric. In our previous paper \cite{sung5}, we considered the case when a 4-manifold $M$ and its Spin$^c$ structure $\frak{s}$ admit a smooth action by a compact Lie group $G$, and defined a $G$-\emph{monopole class} as an element of $H^2(M;\Bbb Z)$ which is the first Chern class of a $G$-equivariant Spin$^c$ structure for which the Seiberg-Witten equations admit a $G$-invariant solution for every $G$-invariant Riemannian metric of $M$.

In order to detect a $G$-monopole class, we need to compute $G$-monopole invariants obtained by the intersection theory on the moduli spaces of $G$-invariant solutions of the Seiberg-Witten equations, and $G$-Bauer-Furuta invariant given by the homotopy class of the Seiberg-Witten map between the subspaces of $G$-invariant configurations. We respectively denote the $G$-monopole invariant and the $G$-Bauer-Furuta invariant of $(M,\frak{s})$ by  $SW^{G}_{M,\frak{s}}$ and  $BF^{G}_{M,\frak{s}}$. If $G=\{e\}$, then they are just the ordinary invariants $SW_{M,\frak{s}}$ and  $BF_{M,\frak{s}}$.

In fact, $G$-monopole classes we have in mind in this paper are the cases when $G$ is finite. Suppose that a compact connected Lie group $G$ of nonzero dimension acts effectively on a smooth closed manifold $M$. If $G$ is not a torus $T^k$, then $G$ contains a Lie subgroup isomorphic to $S^3$ or $S^3/\Bbb Z_2$, and hence $M$ admits a $G$-invariant metric of positive scalar curvature by the well-known Lawson-Yau theorem \cite{law-yau}. (In its original form, the theorem only states that $M$ carries a metric of positive scalar curvature, but one can check that their method can yield a $G$-invariant such metric.) If this is the case for a 4-manifold $M$ with $b_2^G(M)>1$, then $M$ usually has no $G$-monopole class.

On the other hand, in case of a torus action, it is reduced to an $S^1$ action. The Seiberg-Witten invariants of a 4-manifold with an effective $S^1$ action were extensively studied by S. Baldridge \cite{bal1, bal2, bal3}. He showed that if the action has fixed points, the Seiberg-Witten invariants vanish for all Spin$^c$ structures, and if the action is fixed-point free, then the Seiberg-Witten invariants can be read from the those of the quotient 3-orbifold.

In our previous paper, some nontrivial examples of $G$-monopole classes for a finite cyclic group $G$ were shown :
\begin{thm}\cite{sung5}\label{firstth}
Let $M$ and $N$ be  smooth closed oriented 4-manifolds satisfying
$b_2^+(M)> 1$ and $b_2^+(N)=0$, and $\bar{M}_k$ for any $k\geq 2$ be the connected sum
$M\#\cdots \#M\# N$ where there are $k$ summands of $M$.

Suppose that $N$ admits a smooth
orientation-preserving $\Bbb Z_k$-action with at least one free
orbit  such that there exist a $\Bbb Z_k$-invariant Riemannian metric of positive scalar curvature and a $\Bbb Z_k$-equivariant Spin$^c$ structure $\frak{s}_N$ with $c_1^2(\frak{s}_N)=-b_2(N)$.

Define a $\Bbb Z_k$-action on $\bar{M}_k$ induced from that of $N$
and the cyclic permutation of  the $k$ summands of $M$ glued along
a free orbit in $N$, and let $\bar{\frak{s}}$ be the
Spin$^c$ structure on $\bar{M}_k$ obtained by gluing $\frak{s}_N$ and a Spin$^c$
structure $\frak{s}$ of $M$.

Then for any $\Bbb Z_k$-action on $\bar{\frak{s}}$ covering the above
$\Bbb Z_k$-action on $\bar{M}_k$, $SW^{\Bbb
Z_k}_{\bar{M}_k,\bar{\frak{s}}}$ mod 2 is nontrivial if
$SW_{M,\frak{s}}$ mod 2 is nontrivial, and also $BF^{\Bbb
Z_k}_{\bar{M}_k,\bar{\frak{s}}}$ is nontrivial, if
$BF_{M,\frak{s}}$ is nontrivial.
\end{thm}

C. LeBrun and his collaborators \cite{LB3, LB4, IL1, IL2, IS, sung3, ioana} used monopole classes to derive topological obstructions to the existence of an Einstein metric, improving the Hitchin-Thorpe inequality \cite{hit, tho} $$2\chi(X)\pm 3\tau (X)\geq 0$$ which holds on any smooth closed oriented Einstein 4-manifold $X$.
More generally, as observed by Fang, Zhang, and Zhang \cite{FZZ}, these topological obstructions for Einstein metrics on 4-manifolds can be extended to the obstructions to the existence of a quasi-nonsingular solution of the normalized Ricci flow $$\frac{\partial g}{\partial t}=-2 Ric_{g} +\frac{2\bar{s}_g}{n}g,$$
where $\bar{s}_g$ is the average scalar curvature $\frac{\int_M s_g d\mu_g}{\int_M d\mu_g}$ of $g(t)$, and $n$ is the dimension of the manifold.
Following M. Ishida \cite{Ishida}, we say that a smooth solution $\{g(t)|t\in [0,T)\}$ to the normalized Ricci flow  is called \emph{quasi-nonsingular} if $$T=\infty,\ \ \ \textrm{and}\ \ \ \sup_{t\in [0,\infty)} |s_{g(t)}| < \infty.$$

The purpose of this paper is to use $G$-monopole classes to
derive a lower bound of the $L^2$-norm of the Ricci curvature of a $G$-invariant Riemannian metric, even when there is no monopole class, and give a new topological obstruction to the existence of a quasi-nonsingular solution of the normalized Ricci flow for any $G$-invariant initial metric. In particular, this implies the nonexistence of a $G$-invariant Einstein metric or an orbifold Einstein metric on its quotient orbifold.

For instance, we show that a certain connected sum $m\Bbb CP_2 \# n\overline{\Bbb CP}_2$ admits an infinite family of topologically equivalent but smoothly distinct non-free actions of $\Bbb Z_k\oplus H$, where $k\geq 2$ is any integer, and $H$ is any nontrivial finite group acting freely on $S^3$ such that it does not admit a quasi-nonsingular solution of the normalized Ricci flow for any initial metric invariant under such an action.
These examples are different from those of M. Ishida and I. \c{S}uvaina \cite{IS, ioana} who showed that certain connected sums $$m\Bbb CP_2\#n\overline{\Bbb CP}_2\ \ \ \textrm{and}\ \ \ m(S^2\times S^2)\# nK3$$ admit infinitely many free actions of a finite cyclic group and no Einstein metrics invariant under such an action. (In the first examples, also no nonsingular solutions to the normalized Ricci flow invariant under such an action.) Their actions are free so that one can pass to their quotient manifolds and apply ordinary Seiberg-Witten invariants to show the non-existence of Einstein metrics and nonsingular solutions to the normalized Ricci flow on them.

We also apply the above theorem to compute the $\Bbb Z_k$-Yamabe invariant of
$\bar{M}_k$, which is roughly the $\Bbb Z_k$-equivariant version of the Yamabe invariant constructed
using only metrics invariant under a $\Bbb Z_k$ action, and the orbifold Yamabe invariant of $\bar{M}_k/\Bbb Z_k=M\# N/\Bbb Z_k$.

When the $G$ action is finite with isolated fixed points, one may try to find a $G$-monopole class by searching for an ordinary monopole class on the quotient 4-orbifold. But the Seiberg-Witten theory on a 4-orbifold is not fully developed yet, and the readers are referred to \cite{chen1, chen2}.

\section{Ricci flow and $G$-monopole class}
In this section, $G$ denotes a compact Lie group.
\begin{thm}\label{main}
Let $X$ be a smooth closed oriented $4$-manifold with a smooth
$G$-action. Suppose that $c_1(\frak{s})$ is a $G$-monopole class on $X$.
Then for any $G$-invariant Riemannian metric $g$ on $X$,
$$\frac{1}{4\pi^2}\int_X (\frac{s_g^2}{24}+2|W_g^+|^2)\ d\mu_g\geq \frac{2}{3}(c_1^+(\frak{s}))^2,$$ and
$$\frac{1}{8\pi^2}\int_X |Ric_g|^2 d\mu_g\geq 2(c_1^+(\frak{s}))^2-(2\chi(X)+3\tau(X)),$$
where $s_g, W_g^+$, and $Ric_g$  are respectively the scalar
curvature, self-dual Weyl curvature, and Ricci curvature of $g$,
and $c_1^+$, $\chi$, and $\tau$ respectively denote the self-dual harmonic part of $c_1$ with respect to $g$, Euler characteristic, and signature.
\end{thm}
\begin{proof}
As usual, we denote the conformal class of $g$ by $[g]$. The proof is done using LeBrun's Bochner-type argument in the same way as the case of $G=\{1\}$ in \cite{LB4},
 where one needed $c_1(\frak{s})$ to be a monopole class in order to guarantee that a metric
 $\hat{g}\in [g]$ with constant ``\emph{modified scalar curvature}" $s-\sqrt{6}|W^+|$ admits a solution of the Seiberg-Witten equations for $\frak{s}$.

Here $c_1(\frak{s})$ is assumed to be a $G$-monopole class, and hence it's enough to prove that $\hat{g}$ is also $G$-invariant.
Noting that  $\hat{g}$ is a minimizer for the
Yamabe-type functional $$\frak{Y}(\tilde{g}):=\frac{\int_X(s_{\tilde{g}}-\sqrt{6}|W^+_{\tilde{g}}|)\
d\mu_{\tilde{g}}}{(\textrm{Vol}_{\tilde{g}})^{\frac{1}{2}}}$$ defined on $[g]$ of $g$, and the modified scalar curvature of $\hat{g}$ is nonpositive, $\hat{g}$ is unique up to constant
multiplication,  as shown in \cite{sung4}. Since $g$ is invariant under the $G$-action, $\hat{g}$ is
pulled-back under the $G$-action only to a constant multiple of $\hat{g}$, which should
be $\hat{g}$ itself, because the total volume remains unchanged
under the group action.
\end{proof}

\begin{cor}\label{main1}
Let $X$ be a smooth closed oriented $4$-manifold with a smooth
$G$-action. Suppose that $c_1(\frak{s})$ is a $G$-monopole class on $X$, and $X$ admits a $G$-invariant Einstein metric $g$. Then $$2\chi(X)+3\tau (X)\geq \frac{2}{3}(c_1^{+}(\frak{s}))^2\geq \frac{2}{3}c_1^2(\frak{s}).$$
\end{cor}
\begin{proof}
Using the Chern-Gauss-Bonnet theorem and the fact that the trace-free part $\stackrel{\circ}r_{g}$ of $Ric_g$ is zero,
\begin{eqnarray*}
2\chi(X)+3\tau (X)&=& \frac{1}{4\pi^2}\int_X (\frac{s_{g}^2}{24}+2|W_{g}^+|^2-\frac{|\stackrel{\circ}r_{g}|^2}{2})\ d\mu_{g}\\ &=& \frac{1}{4\pi^2}\int_X (\frac{s_{g}^2}{24}+2|W_{g}^+|^2)\ d\mu_{g}\\ &\geq& \frac{2}{3}(c_1^{+}(\frak{s}))^2\\ &\geq& \frac{2}{3}c_1^{2}(\frak{s}),
\end{eqnarray*}
where the first inequality is due to Theorem \ref{main}, and the second one obviously comes from that any 2-form $\alpha$ on $X$ has an orthogonal decomposition $\alpha^++\alpha^-$ into self-dual and anti-self-dual forms so that $$\int_X\alpha\wedge\alpha=\int_X(\alpha^+\wedge\alpha^++\alpha^-\wedge\alpha^-)=||\alpha^+||_{L^2}^2-||\alpha^-||_{L^2}^2.$$
\end{proof}

\begin{thm}\label{main2}
Let $X$ be a smooth closed oriented $4$-manifold with a smooth
$G$-action. Suppose that $c_1(\frak{s})$ is a $G$-monopole class on $X$, and $X$ admits a quasi-nonsingular solution $\{g(t)|t\geq 0\}$ of the normalized Ricci flow for a $G$-invariant initial metric such that
\begin{eqnarray*}
\liminf_{t\rightarrow \infty}(c_1^{+_t}(\frak{s}))^2>0.
\end{eqnarray*}
Then $$2\chi(X)+3\tau (X)\geq \frac{2}{3}\liminf_{t\rightarrow \infty}(c_1^{+_t}(\frak{s}))^2\geq \frac{2}{3}c_1^2(\frak{s}),$$ where $c_1^{+_t}$ is the self-dual harmonic part of $c_1$ with respect to $g(t)$.
\end{thm}
\begin{proof}
We claim that there exists a constant $c>0$ such that
\begin{eqnarray}\label{kurak}
\breve{s}_{g(t)}:=\min_{x\in X}s_{g(t)}(x)< -c.
\end{eqnarray}
Suppose not. Then for any $\epsilon>0$, there exists $t_{\epsilon}>0$ such that $$\breve{s}_{g(t_{\epsilon})}\geq -\epsilon .$$ Note that $g(t)$ for any $t$ is also $G$-invariant by the uniqueness of the Ricci flow. Thus there exists a solution $(A_t, \Phi_t)$ of the Seiberg-Witten equations for $(X,\frak{s})$ with respect to $g(t)$. Thus
\begin{eqnarray*}
(c_1^{+_t}(\frak{s}))^2&\leq& \frac{1}{4\pi^2}\int_X |F_{A_t}^{+_t}|^2 d\mu_{g(t)}\\ &=& \frac{1}{4\pi^2}\int_X |\Phi_t\otimes\Phi_t^*-\frac{|\Phi_t|^2}{2}\textrm{Id}|^2 d\mu_{g(t)}\\ &=& \frac{1}{4\pi^2}\int_X \frac{|\Phi_t|^4}{8} d\mu_{g(t)}.
\end{eqnarray*}
By the well-known Weitzenb\"{o}ck argument, $$|\Phi_t|\leq \max \{-\breve{s}_{g(t)},0 \},$$ and hence $|\Phi_{t_{\epsilon}}|\leq \epsilon.$
Therefore $$(c_1^{+_{t_\epsilon}}(\frak{s}))^2\leq \frac{1}{4\pi^2}\int_X\frac{\epsilon^4}{8} d\mu_{g(t)}=\frac{\epsilon^4}{32\pi^2}\int_Xd\mu_{g(0)},$$ because the normalized Ricci flow preserves the volume. Since $\epsilon>0$ is arbitrary, this yields a contradiction to $\liminf_{t\rightarrow \infty}(c_1^{+_t}(\frak{s}))^2>0.$

By Lemma 3.1 of \cite{FZZ}, any quasi-nonsingular solution satisfying (\ref{kurak}) on a closed manifold must have that
\begin{eqnarray}\label{allsave}
\int_0^\infty \int_X |\stackrel{\circ}r_{g(t)}|^2d\mu_{g(t)}< \infty.
\end{eqnarray}
Then by the Chern-Gauss-Bonnet theorem combined with Theorem \ref{main},
\begin{eqnarray*}
2\chi(X)+3\tau (X)&=& \int_{m}^{m+1}(2\chi(X)+3\tau (X))\ dt\\ &=& \frac{1}{4\pi^2}\int_{m}^{m+1}\int_X (\frac{s_{g(t)}^2}{24}+2|W_{g(t)}^+|^2-\frac{|\stackrel{\circ}r_{g(t)}|^2}{2})\ d\mu_{g(t)}dt\\ &\geq& \liminf_{m\rightarrow\infty}\frac{1}{4\pi^2}\int_{m}^{m+1}\int_X (\frac{s_{g(t)}^2}{24}+2|W_{g(t)}^+|^2-\frac{|\stackrel{\circ}r_{g(t)}|^2}{2})\ d\mu_{g(t)}dt\\ &\geq& \liminf_{m\rightarrow \infty}\frac{1}{4\pi^2}\int_{m}^{m+1}\int_X (\frac{s_{g(t)}^2}{24}+2|W_{g(t)}^+|^2)\ d\mu_{g(t)}dt\\ &\geq& \liminf_{m\rightarrow \infty}\int_m^{m+1}\frac{2}{3}(c_1^{+_t}(\frak{s}))^2dt\\ &=& \liminf_{m\rightarrow \infty}\int_0^{1}\frac{2}{3}(c_1^{+_{m+t}}(\frak{s}))^2dt\\ &\geq& \int_0^{1}\frac{2}{3}\liminf_{m\rightarrow \infty}(c_1^{+_{m+t}}(\frak{s}))^2dt\\&\geq& \frac{2}{3}\liminf_{t\rightarrow \infty}(c_1^{+_t}(\frak{s}))^2,
\end{eqnarray*}
where the 2nd inequality from the last is due to Fatou's lemma.
\end{proof}

\begin{rmk}
The assumption that $\liminf_{t\rightarrow \infty}(c_1^{+_t}(\frak{s}))^2>0$ was needed only to get (\ref{kurak}), and so it can be replaced by the condition $$Y_G(X)<0$$ on the $G$-Yamabe invariant of $X$, which will be introduced in the following section.
\end{rmk}

We now produce some non-existence examples of $G$-invariant Einstein metrics or more generally quasi-nonsingular solutions of the normalized Ricci flow, where the Hitchin-Thorpe inequality is satisfied while there may not exist any monopole class.
\begin{thm}\label{grace}
Let $M$, $N$, and $\bar{M}_k$ be as in Theorem \ref{firstth}.  Suppose that $M$ has  nonzero mod 2 Seiberg-Witten invariant
for a Spin$^c$ structure $\frak{s}$, and $$0<2\chi(M)+3\tau (M)< \frac{1}{k}(12(k-1)+12b_1(N)+3b_2(N)).$$ Then $\bar{M}_k$ does not admit a quasi-nonsingular solution to the normalized Ricci flow for any $\Bbb Z_k$-invariant initial metric. In particular, $\bar{M}_k$ never admits a $\Bbb Z_k$-invariant Einstein metric, and $M\# N/{\Bbb Z_k}$ never admits an orbifold Einstein metric.
\end{thm}
\begin{proof}
First note that $\bar{M}_k$ admits a $\Bbb Z_k$-invariant Einstein metric iff $\bar{M}_k/{\Bbb Z_k} = M\# N/{\Bbb Z_k}$ admits an orbifold Einstein metric, when $N/{\Bbb Z_k}$ is an orbifold. Because an Einstein metric is a static solution of the normalized Ricci flow, we will prove only the first statement.

Think of $\bar{M}_k$ as the connected sum $kM \# N$, and let $\frak{s}_1$ and $\frak{s}_2$ be the restriction of $\bar{\frak{s}}$ to $kM-B^4$ and $N-B^4$ respectively,
where $B^4$ is a small open ball for the connected sum operation. Then $$c_1(\bar{\frak{s}})=c_1(\frak{s}_1)+c_1(\frak{s}_2)\in H^2(kM-B^4)\oplus H^2(N-B^4)=H^2(\bar{M}_k),$$ and with respect to any Riemannian metric on $\bar{M}_k$
\begin{eqnarray*}
(c_1^{+}(\bar{\frak{s}}))^2&=& (c_1^{+}(\frak{s}_1)+c_1^{+}(\frak{s}_2))^2\\ &=&(c_1^{+}(\frak{s}_1))^2+2c_1^{+}(\frak{s}_1)\cdot c_1^{+}(\frak{s}_2)+(c_1^{+}(\frak{s}_2))^2\\ &\geq& (c_1^{+}(\frak{s}_1))^2+2c_1^{+}(\frak{s}_1)\cdot c_1^{+}(\frak{s}_2).
\end{eqnarray*}

\begin{lem}
$- \frak{s}_{N}:=\frak{s}_{N}\otimes (-\det(\frak{s}_{N}))$ is also $\Bbb Z_k$-equivariant.
\end{lem}
\begin{proof}
Since $\frak{s}_{N}$ is $\Bbb Z_k$-equivariant, so is its associated determinant line bundle $\det(\frak{s}_{N})$. Therefore $\frak{s}_{N}\otimes (-\det(\frak{s}_{N}))$ is also $\Bbb Z_k$-equivariant.
\end{proof}

Let $\bar{\frak{s}}'$ be the Spin$^c$ structure on $\bar{M}_k$ replacing $\frak{s}_{N}$ in $\bar{\frak{s}}$ by $-\frak{s}_{N}$, and  $\frak{s}_1'$ and $\frak{s}_2'$ be defined as above. Then  $\frak{s}_1' =\frak{s}_1$ and $c_1(\frak{s}_2')=-c_1(\frak{s}_2)$. Therefore we have either  $$ c_1^{+}(\frak{s}_{1})\cdot c_1^{+}(\frak{s}_{2})\geq 0,$$ or  $$ c_1^{+}(\frak{s}_{1}')\cdot c_1^{+}(\frak{s}_{2}')\geq 0.$$ In the first case,
\begin{eqnarray*}
(c_1^{+}(\bar{\frak{s}}))^2 &\geq&
 (c_1^{+}(\frak{s}_{1}))^2\\
&\geq&  c_1^2(\frak{s}_{1})\\
&=& k \ c_1^2(\frak{s})\\
&\geq & k(2\chi(M)+3\tau (M)),
\end{eqnarray*}
where the last inequality holds because the Seiberg-Witten moduli space of $\frak{s}$ on $M$ has nonnegative dimension.
Likewise in the second case, $$(c_1^{+}(\bar{\frak{s}}'))^2 \geq k(2\chi(M)+3\tau (M)).$$

Now let's assume to the contrary that $\bar{M}_k$ does admit such a solution $\{g(t)|t\geq 0\}$ of the normalized Ricci flow.
We claim that there exists a constant $c>0$ such that
\begin{eqnarray*}
\breve{s}_{g(t)}:=\min_{x\in X}s_{g(t)}(x)< -c.
\end{eqnarray*}
By Theorem \ref{firstth}, both $\bar{\frak{s}}$ and $\bar{\frak{s}}'$ are $\Bbb Z_k$-monopole classes on $\bar{M}_k$. If there does not exist such $c>0$, then by the same method as in Theorem \ref{main2}, for any $\epsilon>0$, there exists $t_\epsilon >0$ such that
$$(c_1^{+_{t_\epsilon}}(\bar{\frak{s}}))^2\leq \frac{\epsilon^4}{32\pi^2}\int_Xd\mu_{g(0)},$$  and
$$(c_1^{+_{t_\epsilon}}(\bar{\frak{s}}'))^2\leq \frac{\epsilon^4}{32\pi^2}\int_Xd\mu_{g(0)},$$ both of which together imply
$$k(2\chi(M)+3\tau (M))\leq \frac{\epsilon^4}{32\pi^2}\int_Xd\mu_{g(0)}.$$ By the assumption $2\chi(M)+3\tau (M)>0$, the claim is justified, and we obtain (\ref{allsave}) by Lemma 3.1 of \cite{FZZ}.

Then proceeding as in the last part in the proof of Theorem \ref{main2} we get
\begin{eqnarray*}
2\chi(\bar{M}_k)+3\tau (\bar{M}_k)&\geq& \liminf_{m\rightarrow \infty}\frac{1}{4\pi^2}\int_{m}^{m+1}\int_X (\frac{s_{g(t)}^2}{24}+2|W_{g(t)}^+|^2)\ d\mu_{g(t)}dt \\ &\geq&
\liminf_{m\rightarrow \infty}\int_m^{m+1}\max(\frac{2}{3}(c_1^{+_t}(\bar{\frak{s}}))^2, \frac{2}{3}(c_1^{+_t}(\bar{\frak{s}}'))^2)dt\\ &\geq&
\frac{2}{3}k(2\chi(M)+3\tau (M)).
\end{eqnarray*}
A simple computation gives
\begin{eqnarray*}
2\chi(\bar{M}_k)+3\tau (\bar{M}_k)&=& k(2\chi(M)+3\tau (M))+2\chi(N)+3\tau (N)-4k\\
&=& k(2\chi(M)+3\tau (M))+4(1-k)-4b_1(N)-b_2(N).
\end{eqnarray*}
Plugging this into the above gives $$k(2\chi(M)+3\tau (M))+4(1-k)-4b_1(N)-b_2(N)\geq \frac{2}{3}k(2\chi(M)+3\tau (M))$$ which simplifies to $$\frac{k}{3}(2\chi(M)+3\tau (M))\geq 4(k-1)+4b_1(N)+b_2(N),$$  yielding a contradiction.
\end{proof}

The following lemma is a slight generalization of \cite[Theorem 7.1]{sung5}.
\begin{lem}\label{repenter}
Let $M$  be a smooth closed  4-manifold and $\{M_i|i\in \frak I\}$ be a family of smooth 4-manifolds such that every $M_i$ is homeomorphic to $M$ and the numbers of mod 2 basic classes of $M_i$'s are all mutually different, but each $M_i\#l_i(S^2\times S^2)$ is diffeomorphic to $M\#l_i(S^2\times S^2)$ for an integer $l_i\geq 1$.

If $l_{max}:=\sup_{i\in\frak I}l_{i}<\infty$, then for any integers $k\geq 2, n\geq 0$, and $l\geq l_{max}+1$, $$X:=klM\# kln\overline{\Bbb CP}_2 \#(l-1)(S^2\times S^2)$$
admits an $\frak I$-family of topologically equivalent but smoothly
distinct non-free actions of $\Bbb Z_k\oplus H$ where $H$ is any
group of order $l$ acting freely on $S^3$.
\end{lem}
\begin{proof}
The following proof for $n\geq 1$ cases is almost parallel to the $n=0$ case of \cite{sung5}, and first recall that $(l-1)(S^2\times S^2)$ admits an $H$ action defined as the deck transformation
map of the $l$-fold covering map onto $\widehat{S^1\times L}$ for $L=S^3/H$, where $\widehat{S^1\times L}$ is the manifold obtained from the surgery on $S^1\times L$ along an $S^1\times \{\textrm{pt}\}$.

Think of $X$ as $$klM_i\# kln\overline{\Bbb CP}_2\#(l-1)(S^2\times
S^2),$$ on which  $H$ acts as the deck transformation
map of the $l$-fold covering map onto
$$\bar{M}_{i,k}:=kM_i\# kn\overline{\Bbb CP}_2\#\widehat{S^1\times L}.$$ To
define a $\Bbb Z_k$-action, note that $\bar{M}_{i,k}$ has a $\Bbb Z_k$-action coming from the $\Bbb Z_k$-action of $ kn\overline{\Bbb CP}_2\#\widehat{S^1\times L}$ defined in \cite[Theorem 6.4]{sung5}, which is basically a
rotation along the $S^1$-direction, and whose fixed point set is $\{0\}\times S^2$ in the attached $D^2\times S^2$. This $\Bbb Z_k$ action is
obviously lifted to the above $l$-fold cover, and it commutes
with the above defined $H$ action. Thus we have an $\frak I$-family
of $\Bbb Z_k\oplus H$ actions on $X$,
which are all topologically equivalent by using the homeomorphism
between each $M_i$ and $M$.

Recall from \cite[Theorem 6.4]{sung5} that all the Spin$^c$ structures on
$\widehat{S^1\times L}$ are $\Bbb Z_k$-equivariant and satisfy
$c_1^2=-b_2(\widehat{S^1\times L})=0$. Let $\bar{\frak{s}}_i$ be the $\Bbb Z_k$-equivariant
Spin$^c$ structure on $\bar{M}_{i,k}$ obtained by gluing a Spin$^c$
structure $\frak{s}_i$ of $M_i$ and a $\Bbb Z_k$-equivariant Spin$^c$ structure $\frak{s}_N$ on $\widehat{S^1\times L}\# kn\overline{\Bbb CP}_2$ satisfying $$c_1^2(\frak{s}_N)=-b_2(\widehat{S^1\times L}\# kn\overline{\Bbb CP}_2)=-kn.$$ By \cite[Theorem 4.2]{sung5} and the fact that $b_1(\widehat{S^1\times L}\#kn\overline{\Bbb CP}_2)=0$, for
any Spin$^c$ structure ${\frak{s}}_i$ on $M_i$,
$$SW^{\Bbb Z_k}_{\bar{M}_{i,k},\bar{\frak{s}}_i}\equiv SW_{M_i,\frak{s}_i}\ \ \ \textrm{mod}\ 2,$$
and hence 
the corresponding Seiberg-Witten polynomials satisfy $$SW^{\Bbb
Z_k}_{\bar{M}_{i,k}}\equiv SW_{M_i}\sum_{[\alpha],[\beta]}[\alpha][\beta]$$ modulo 2, where $[\alpha]$ runs through any element of $H^2(\widehat{S^1\times L};\Bbb Z)=H_1(L;\Bbb Z)$ and $[\beta]$ runs through any $\Bbb Z_k$-equivariant element of $H^2(kn\overline{\Bbb CP}_2;\Bbb Z)$ satisfying that $[\beta]^2=-kn,$ and $[\beta]$ restricts to a generator of the 2nd cohomology in each $\overline{\Bbb CP}_2$-summand. Therefore $SW^{\Bbb
Z_k}_{\bar{M}_{i,k}}$ mod 2 for all $i$ have mutually different numbers of monomials, and hence the $\frak I$-family of $\Bbb Z_k\oplus H$ actions on $X$ cannot
be smoothly equivalent, completing the proof.
\end{proof}
\begin{thm}\label{LLL}
Let  $k\geq 2$ be an integer and $H$ be a finite group of order $l\geq 2$ acting freely on $S^3$. Then  for infinitely many $m\in \Bbb Z^+$ and any integer $n> \frac{4m-2}{3}$,
the manifold $$(klm+l-1)\Bbb CP_2 \# (kl(m+n)+l-1)\overline{\Bbb CP}_2$$ admits an infinite family of topologically equivalent but smoothly distinct non-free actions of $\Bbb Z_k\oplus H$
such that it admits no quasi-nonsingular solution to the normalized Ricci flow for any initial metric invariant under such an action.
\end{thm}
\begin{proof}
Note that the above manifold is diffeomorphic to $$Y:=(klm+l-1)(S^2\times S^2) \# kln\overline{\Bbb CP}_2.$$ As in \cite[Corollary 7.2]{sung5}, we use the construction of B. Hanke, D. Kotschick, and J. Wehrheim \cite{Kot}, which shows that
$m(S^2\times S^2)$ for infinitely many $m$ has the property of $M$ in Lemma \ref{repenter} with each $l_i=1$ and $|\frak I|=\infty$.
Using Lemma \ref{repenter}, $Y$ admits such  $\Bbb Z_k\oplus H$ actions so that
\begin{eqnarray*}
Y/H=kM_i\# kn\overline{\Bbb CP}_2 \# \widehat{S^1\times L}
\end{eqnarray*}
for $\{M_i|i\in \frak I\}$.

Thus we only need to show that $Y/H$ does not admit a quasi-nonsingular solution to the normalized Ricci flow invariant under the $\Bbb Z_k$ action.
As proven in \cite[Theorem 6.4]{sung5}, $ kn\overline{\Bbb CP}_2 \# \widehat{S^1\times L}$ is an example of $N$ in Theorem \ref{firstth}, and hence we can apply Theorem \ref{grace} to $Y/H$.
Using $b_1(\widehat{S^1\times L})=b_2(\widehat{S^1\times L})=0$, 
a simple computation gives
\begin{eqnarray*}
\frac{12(k-1)+12b_1(N)+3b_2(N)}{k} &=& \frac{12(k-1)}{k}+3n\\ &\geq& 6+3n\\ &>& 4m+4\\ &=& 2\chi(M_i)+3\tau (M_i),
\end{eqnarray*}
which completes the proof.
\end{proof}

\begin{rmk}
Just for a simple remark, the above manifold of Theorem \ref{LLL} satisfies the Hitchin-Thorpe inequality when $n\leq 4(m+\frac{1}{k})$.

Also note that it obviously admits a metric of positive scalar curvature. But such metrics are never invariant under those $\Bbb Z_k\oplus H$ actions, because it has nontrivial $\Bbb Z_k\oplus H$ monopole invariant. C. LeBrun \cite{LB6} was the first who discovered that a finite group acts freely on certain connected sums $m\Bbb CP_2 \# n\overline{\Bbb CP}_2$ so that there exist no metrics of positive scalar curvature invariant under the action, and moreover the quotient manifolds have negative Yamabe invariants. On the other hand, D. Ruberman \cite{ruberman} showed that for any $m\geq 2$ and $n>10m$, the space of positive scalar curvature metrics on $2m\Bbb CP_2 \# n\overline{\Bbb CP}_2$ has infinitely many components.

For an example of $H$, one can take $\Bbb Z_l$ for $(k,l)=1$ so that $\Bbb Z_k\oplus H$ is isomorphic to $\Bbb Z_{kl}$.
\end{rmk}

\section{Computation of $G$-Yamabe invariant and orbifold Yamabe invariant}
\setcounter{equation}{0}
When a smooth closed $n$-manifold $X$ admits a smooth group action by a compact Lie group $G$, the $G$-Yamabe invariant can be defined in an analogous way to the ordinary Yamabe invariant. For a $G$-invariant Riemannian metric $g$ on $X$, we let $[g]_G$ be the set of smooth $G$-invariant metrics conformal to $g$. Following \cite{sung11},  define the \emph{$G$-Yamabe constant} of $(X,[g]_G)$ as
\begin{eqnarray}\label{yam}
Y(X,[g]_G):=\inf_{\hat{g}\in [g]_G}\frac{\int_X s_{\hat{g}}\
dV_{\hat{g}}}{(\int_X dV_{\hat{g}})^{\frac{n-2}{n}}},
\end{eqnarray}
 and the \emph{$G$-Yamabe invariant} of $X$ as $$Y_G(X):=\sup_{[g]_G} Y(X,[g]_G).$$  When the $G$-action is trivial, $Y(X,[g]_G)$ and $Y_G(X)$ are obviously  the ordinary Yamabe constant $Y(X,[g])$ and the ordinary Yamabe invariant $Y(X)$ respectively.

By the result of E. Hebey and M. Vaugon \cite{HV}, the $G$-equivariant Yamabe problem can be solved for $n\geq 3$ by minimizing the Yamabe functional defined on each $[g]_G$. The minimizers have constant scalar curvature, and there exists the Aubin-type inequality $$Y(X,[g]_G)\leq \Lambda_n (\inf_{x\in X}|Gx|)^{\frac{2}{n}},$$ where $\Lambda_n$ defined as $n(n-1)(\textrm{Vol}(S^n(1)))^{\frac{2}{n}}$ is the Yamabe invariant $Y(S^n)$ of $S^n$, and $|Gx|$ denotes the cardinality of the orbit of $x$.

When $Y(X,[g]_G)\leq 0$, by definition $$Y(X,[g])\leq Y(X,[g]_G)\leq  0,$$ and hence an ordinary Yamabe minimizer in $[g]$  must also be a $G$-Yamabe minimizer, because the metrics with nonpositive constant scalar curvature are unique up to constant in a conformal class so that they are also $G$-invariant. Thus in that case
\begin{eqnarray*}
Y(X,[g]_G)&=&Y(X,[g]).
\end{eqnarray*}

We present some practical formulae for computing $Y(X,[g]_G)$ and $Y_G(X)$, which are exactly the same forms as the ordinary Yamabe case.
\begin{prop}\label{repent}
Let $r\in [\frac{n}{2},\infty]$. Then
\begin{eqnarray*}
|Y(X,[g]_G)|&=& \inf_{\tilde{g}\in[g]_G}(\int_X |s_{\tilde{g}}|^r
d\mu_{\tilde{g}})^{\frac{1}{r}}(\textrm{Vol}_{\tilde{g}})^{\frac{2}{n}-\frac{1}{r}}\\ &=& \inf_{\tilde{g}\in[g]_G}(\int_X |s^-_{\tilde{g}}|^r
d\mu_{\tilde{g}})^{\frac{1}{r}}(\textrm{Vol}_{\tilde{g}})^{\frac{2}{n}-\frac{1}{r}}\ \ \ \textrm{if}\ \ Y(X,[g]_G)\leq 0,
\end{eqnarray*}
where the infimums are realized only by the minimizer in (\ref{yam}), and $s_{\tilde{g}}^-$ is defined as $\min(s_{\tilde{g}},0)$.

If $Y_G(X)\leq 0$,
\begin{eqnarray*}
Y_G(X)&=& -\inf_{g\in \mathcal{M}_G}(\int_X |s_{g}|^r
d\mu_{g})^{\frac{1}{r}}(\textrm{Vol}_{g})^{\frac{2}{n}-\frac{1}{r}}\\
&=& -\inf_{g\in \mathcal{M}_G}(\int_X |s_{g}^-|^r
d\mu_{g})^{\frac{1}{r}}(\textrm{Vol}_{g})^{\frac{2}{n}-\frac{1}{r}},\nonumber
\end{eqnarray*}
where $\mathcal{M}_G$ is the space of all smooth $G$-invariant Riemannian metrics on $X$.
\end{prop}
\begin{proof}
If $Y(X,[g]_G)> 0$, then it can be proved in the same way as the ordinary Yamabe case. (For a proof, see \cite{sung0}.)

If $Y(X,[g]_G)\leq 0$, then
\begin{eqnarray*}
Y(X,[g]_G)&=& Y(X,[g])\\ &=&-\inf_{\tilde{g}\in[g]}(\int_X |s_{\tilde{g}}|^r
d\mu_{\tilde{g}})^{\frac{1}{r}}(\textrm{Vol}_{\tilde{g}})^{\frac{2}{n}-\frac{1}{r}}\\ &=& -\inf_{\tilde{g}\in[g]}(\int_X |s^-_{\tilde{g}}|^r
d\mu_{\tilde{g}})^{\frac{1}{r}}(\textrm{Vol}_{\tilde{g}})^{\frac{2}{n}-\frac{1}{r}},
\end{eqnarray*}
and these infimums are realized by the Yamabe minimizers, which are $G$-invariant. Therefore it's enough to take the infimums on a smaller set $[g]_G$.

The formulae for $Y_G(X)$ are now straightforward.
\end{proof}

One of the important facts about $G$-Yamabe constant and
$G$-Yamabe invariant is that they are basically equivalent to
the orbifold Yamabe constant and the orbifold Yamabe invariant of the
quotient manifold, when $G$ is finite and $X/G$ is an orbifold.

Let $V$ be a closed orbifold of dimension $n$. For an orbifold Riemannian metric $g$ on
$V$, $[g]_{orb}$ denotes the set of orbifold Riemannian metrics
conformal to $g$. In the same as the ordinary Yamabe problem, K.
Akutagawa and B. Botvinnik \cite{aku} defined the \emph{orbifold
Yamabe constant} $Y(V,[g]_{orb})$ of  $[g]_{orb}$  as the infimum
of the normalized Einstein-Hilbert functional on   $[g]_{orb}$,
and the \emph{orbifold Yamabe invariant}
$$Y_{orb}(V):=\sup_{[g]_{orb}} Y(V,[g]_{orb}).$$ They also obtained the Aubin-type inequality
$$Y(V,[g]_{orb})\leq \min_{1\leq i\leq m}\frac{\Lambda_n}{|\Gamma_i|^{\frac{2}{n}}},$$ where $\{(\check{p}_1,\Gamma_1),\cdots,(\check{p}_m,\Gamma_m)\}$ is the singularity of $V$.

A group action is called \emph{pseudo-free}, if non-free orbits are isolated. For a smooth pseudo-free action on a smooth manifold by a finite group, its quotient space has a natural   orbifold structure.
\begin{thm}\label{orbi-yam}
Let $X$ be a smooth closed $n$-manifold with smooth pseudo-free action by a finite group $G$. Then for an orbifold Riemannian metric $g$ on $X/G$,
$$Y_{orb}(X/G,[g]_{orb})=\frac{Y_G(X, [\pi^*g])}{|G|^{\frac{2}{n}}},\ \  \textrm{and}\ \
Y_{orb}(X/G)=\frac{Y_G(X)}{|G|^{\frac{2}{n}}},$$ where $\pi : X\rightarrow X/G$ is the quotient map.
\end{thm}
\begin{proof}
The proof is obvious from the observation that
$[\pi^*g]_G=\pi^*[g]_{orb}$ and $\pi$ is a branched $|G|$-fold covering.
\end{proof}

In \cite{sung11}, we obtained gluing formulae for the $G$-Yamabe invariant for the surgery of codimension 3 and more, which made it possible to compute some $G$-Yamabe invariants of products of spheres and their connected sums. Here, the existence of a $\Bbb Z_k$-monopole class on $\bar{M}_k$ enables us to compute its $\Bbb Z_k$-Yamabe invariant :
\begin{thm}\label{lastthm}
Let $M$ be a smooth closed oriented 4-manifold with a Spin$^c$ structure $\frak{s}$ satisfying  $$Y(M)=-4\sqrt{2}\pi \sqrt{c_1^2(\frak{s})},$$ and N, $\bar{M}_k$ be as in Theorem \ref{firstth}. Suppose that $\frak{s}$ has nonzero mod 2 Seiberg-Witten  invariant or nontrivial Bauer-Furuta invariant, and the $\Bbb Z_k$-action on $N$ is pseudo-free. Then
$$Y_{\Bbb Z_k}(\bar{M}_k)=\sqrt{k}Y(M),$$ and $$Y_{orb}(M\# N/\Bbb Z_k)=Y(M).$$
\end{thm}
\begin{proof}
First, we show that
$$Y_{\Bbb Z_k}(\bar{M}_k)\geq \sqrt{k}Y(M)$$ by using the standard gluing method of the ordinary Yamabe invariant.

Take a $\Bbb Z_k$-invariant metric of positive scalar curvature on  $N$, and make $k$ cylindrical ends in a $\Bbb Z_k$-symmetric way keeping the positivity of scalar curvature by performing the Gromov-Lawson surgery \cite{GL}. On each $M$ we take a metric $g$ which  approximates the Yamabe invariant of $M$, and  also make a cylindrical end likewise.
By gluing  these pieces, we have a $\Bbb Z_k$-invariant metric on $\bar{M}_k$, denoted by $h$.
For any $\varepsilon > 0$, we can arrange the Gromov-Lawson surgery\footnote{For this, one may consult a refined way of Gromov-Lawson surgery as in \cite{sung1}. Another easy way suggested by C. LeBrun in \cite{lb} is as follows. Let $W\subset M$ be a small ball around the point where the connected sum is performed. One can take a conformal change $\varphi g$ of $g$ such that $\varphi\equiv 1$ outside of $W$ and the scalar curvature of $\varphi g$ is positive on a much smaller open subset $W'$ of $W$, and $$\int_{M}(s_{\varphi g}^-)^2d\mu_{\varphi g}\leq  \int_{M}(s_g^-)^2d\mu_g+\varepsilon^2.$$ Then one can perform an ordinary Gromov-Lawson surgery on $W'\subset (M,\varphi g)$ keeping the positivity of scalar curvature  to achieve the inequality (\ref{GR-LA}). A detailed proof of this method can be found in \cite{sung-hpn}.}  so that $h$ depending  $\varepsilon$ satisfies
\begin{eqnarray}\label{GR-LA}
\int_{\bar{M}_k}(s_h^-)^2d\mu_h\leq k \int_{M}(s_g^-)^2d\mu_g+\frac{\varepsilon}{2}\leq k(Y(M))^2+\varepsilon.
\end{eqnarray}
 Since $\varepsilon >0$ is arbitrary, the application of Proposition \ref{repent}  with $r=2$ yields $$Y_{\Bbb Z_k}(\bar{M}_k)\geq \sqrt{k}Y(M).$$

To prove the reverse inequality we will show
$$\int_{\bar{M}_k} s_{\bar{g}}^2\ d\mu_{\bar{g}} \geq k(Y(M))^2$$
  for any $\Bbb Z_k$-invariant metric $\bar{g}$ on $\bar{M}_k$.
 Since $c_1(\bar{\frak{s}})$ is a $\Bbb Z_k$-monopole class of $\bar{M}_k$, there exists a solution of the Seiberg-Witten equations of $\bar{\frak{s}}$ for $\bar{g}$. Then LeBrun's Weitzenb\"ock-type argument \cite{LB5} gives
\begin{eqnarray*}
\int_{\bar{M}_k} s_{\bar{g}}^2\ d\mu_{\bar{g}} &\geq&  32\pi^2 (c_1^+(\bar{\frak{s}}))^2.
\end{eqnarray*}
Using $(c_1^+(\bar{\frak{s}}))^2\geq k\ c_1^2(\frak{s})$ (or $(c_1^+(\bar{\frak{s}}'))^2\geq k\ c_1^2(\frak{s})$) proved in Theorem \ref{grace}, we get desired
\begin{eqnarray*}
\int_{\bar{M}_k} s_{\bar{g}}^2\ d\mu_{\bar{g}} &\geq& 32\pi^2k\ c_1^2(\frak{s})=
 k(Y(M))^2,
\end{eqnarray*}
which completes the proof of the first statement. Then the second statement follows from Theorem \ref{orbi-yam}.
\end{proof}

In fact, one can easily generalize the above theorem to the statement that for any blow-up $M'$ of such $M$,
$$Y_{\Bbb Z_k}(\bar{M'}_k)=\sqrt{k}Y(M')=\sqrt{k}Y(M),$$ and $$Y_{orb}(M'\# N/\Bbb Z_k)=Y(M')=Y(M).$$

\begin{xpl}
For such an example of $M$ in the above theorem which has nonzero mod 2 Seiberg-Witten invariant, there exists a  minimal compact K\"ahler surface of nonnegative Kodaira dimension with $b_2^+(M)>1$. Certain surgeries along tori
in product manifolds of two Riemann surfaces of genus $>1$ also have such a property. For details, the readers are referred to \cite{sung2}.

But such examples of $M$ with nontrivial Bauer-Furuta invariant are not well understood enough. According to S. Bauer's computation \cite{bau}, if $X_j$ for $j=1,\cdots,4$ are minimal compact K\"ahler surfaces satisfying $$b_1(X_j)=0,\ \ b_2^+(X_j)\equiv 3\  \textrm{mod}\ 4,\ \  \sum_{j=1}^4b_2^+(X_j)\equiv 4\ \textrm{mod}\ 8,$$
then $\#_{j=1}^m X_j$ for each $m=1,\cdots,4$ is such an example of $M$.

Applying the above theorem to such an $M$ and $N=S^4$, we obtain
\begin{eqnarray*}
Y_{orb}(M\#S(L(p;q)))&=&Y_{orb}(M\# S^4/\Bbb Z_p)\\ &=& Y(M),
\end{eqnarray*}
 where $S(L(p;q))$ is the suspension of the Lens space $L(p;q)=S^3/\Bbb Z_p$ with the $\Bbb Z_p$-action given by $(z_1,z_2)\sim (e^{\frac{2\pi i}{p}}z_1,e^{\frac{2\pi iq}{p}}z_2)\in \Bbb C^2$ for coprime integers $p$ and $q$.

More examples of $N$ are given in \cite{sung5}.
\end{xpl}

\begin{rmk}
Just as the ordinary Yamabe invariant is a smooth topological invariant, the orbifold Yamabe invariant can distinguish differential structures of orbifolds. For example, let $M$ be as in the above example and $N$ be as in the above theorem. Suppose further that $M$ is simply connected. The above theorem asserts that $$Y_{orb}(M\#\overline{\Bbb CP}_2\# N/\Bbb Z_k)=Y(M\#\overline{\Bbb CP}_2)=Y(M)\leq 0.$$ On the other hand, $M\#\overline{\Bbb CP}_2$ is nonspin, and hence by Freedman's theorem \cite{freed},  $M\#\overline{\Bbb CP}_2\# N/\Bbb Z_k$ is homeomorphic to $$b_2^+(M)\Bbb CP_2\#(b_2^-(M)+1)\overline{\Bbb CP}_2\# N/\Bbb Z_k,$$ whose orbifold Yamabe invariant is positive. Therefore they are not diffeomorphic as orbifolds.

The ordinary Yamabe invariant $Y(\bar{M}_k)$ of $\bar{M}_k$ is hardly known except for very special cases \cite{IL2}. It seems plausible that it is equal to  $Y_{\Bbb Z_k}(\bar{M}_k)$ under the assumption of Theorem \ref{lastthm}.
\end{rmk}

\bigskip

\noindent{\bf Acknowledgement.} The author would like to express sincere thanks to the anonymous referee for pertinent suggestions.

\end{document}